\documentclass[11pt]{amsart}
\usepackage{amssymb, amsmath, amscd, mathrsfs, graphicx, epsbox, color}

\makeatletter
\def\@tocline#1#2#3#4#5#6#7{\relax
  \ifnum #1>\c@tocdepth % then omit
  \else
    \par \addpenalty\@secpenalty\addvspace{#2}%
    \begingroup \hyphenpenalty\@M
    \@ifempty{#4}{%
      \@tempdima\csname r@tocindent\number#1\endcsname\relax
    }{%
      \@tempdima#4\relax
    }%
    \ifnum#1=1\large\fi%%%%% size for section
    \ifnum#1=2\small\fi%%%%% size for subsection
    \parindent\z@ \leftskip#3\relax \advance\leftskip\@tempdima\relax
    \rightskip\@pnumwidth plus4em \parfillskip-\@pnumwidth
    #5\leavevmode\hskip-\@tempdima #6\nobreak\relax
    \hfil\hbox to\@pnumwidth{\@tocpagenum{#7}}\par
    \nobreak
    \endgroup
  \fi}
\makeatother

\newtheorem{thm}{Theorem}[section]
\newtheorem{lem}[thm]{Lemma}

\newtheorem{prop}[thm]{Proposition}

\theoremstyle{definition}

\newtheorem{rem}[thm]{Remark}
\newtheorem{defn}[thm]{Definition}

\theoremstyle{remark}

\numberwithin{equation}{section}

\def\R{{\mathbb R}}
\def\Z{{\mathbb Z}}
\def\C{{\mathbb C}}
\def\Q{{\mathbb Q}}

\def\G{\Gamma}

\allowdisplaybreaks

\begin{document}
\title[Nontrivial algebraic cycles in
the Jacobian varieties]
{Nontrivial algebraic cycles in\\
the Jacobian varieties of some quotients of Fermat curves}

\author[Yuuki Tadokoro]{Yuuki Tadokoro}
\address{Natural Science Education,
Kisarazu National College of Technology, 2-11-1 Kiyomidai-Higashi,
Kisarazu, Chiba 292-0041, Japan}
\email{tado\char`\@nebula.n.kisarazu.ac.jp}

\maketitle

\begin{abstract}
We obtain the trace map images of the values of certain harmonic volumes for
some quotients of Fermat curves.
These provide the algorithm showing that
the algebraic cycles called by $k$-th Ceresa cycles
are not algebraically equivalent to zero
in the Jacobian varieties.
We apply the method the case for the prime $N < 1000, k=1$
and $N=7,13, k\leq (N-3)/2$.
\end{abstract}

\section{Introduction}
Let $X$ be a compact Riemann surface
of genus $g\geq 2$
and $J(X)$ its Jacobian variety.
By the Abel-Jacobi map $X\to J(X)$,
$X$ is embedded in $J(X)$.
Let $X_k$ be the $k$-th symmetric product of $X$
and
$W_k$ its image of the Abel-Jacobi map.
The algebraic $k$-cycle $W_k-W_k^{-}$ in $J(X)$,
called by $k$-th Ceresa cycle,
is homologous to zero.
Here we denote by $W_k^{-}$ the image of $W_k$
under the multiplication map by $-1$.
If $X$ is hyperelliptic,
$W_k=W_k^{-}$ in $J(X)$.
For the rest of this paper, suppose $g\geq 3$.
We put $X-X^{-}=W_1-W_1^{-}$.
B. Harris \cite{H-3} studied the problem
whether the cycle $X-X^-$
in $J(X)$
is algebraically equivalent to zero or not.
Roughly speaking,
it can be ``continuously'' (algebraically)
deformed into the zero cycle or not.
See \cite{F} for example.
Faucette \cite{Fa} also studied a sufficient condition
that the algebraic cycle $W_k-W_k^{-}$
is not algebraically equivalent to zero in $J(X)$.
We remark that Weil \cite[pp.~331]{We1}
mentioned the homologous zero cycle $W_k-W_k^{-}$
in question.

Let $N$ be a prime number such that $N=1$ modulo 3 and
$m$ be an integer $m^2+m+1=0$ modulo $N$.
For the quotient of Fermat curve $C_{N}=C_{N}^{1,m}$,
we denote $f(N,k)$ by a value of the harmonic volume
which is defined later using the special values of
the generalized hypergeometric function ${}_3F_{2}$.
Using Otsubo's result \cite{O},
we obtain the main theorem
\begin{thm}
For the quotient of Fermat curve $C_N$ and
an integer $k$ such that $1\leq k \leq (N-3)/2$,
if the value $f(N,k)$ is not integer, then
$W_k-W_k^-$ is not algebraically equivalent to zero
in $J(C_{N})$.
\end{thm}

The harmonic volume $I$ for $X$ was introduced by 
Harris \cite{H-1}, using Chen's iterated integrals \cite{C}.
Let
$H$ denote the first integral homology group $H_1(X; \Z)$
of $X$.
The harmonic volume $I$
is defined to be
a homomorphism $(H^{\otimes 3})^\prime\to \R/{\Z}$.
Here $(H^{\otimes 3})^\prime$
is a certain subgroup of $H^{\otimes 3}$.
The twice $2I$ factors through the third exterior product
$\wedge^3 H$,
and we call it the harmonic volume similarly.
See Section \ref{The harmonic volumes} for the definition.
Let $F_{N}$ denote the Fermat curve for $N\in \Z_{\geq 4}$.
Using $I$, Harris \cite{H-3, H-4} proved that
the algebraic cycle
$F_4-F_4^-$ is not algebraically equivalent to zero
in $J(F_4)$.
Ceresa \cite{Ce} showed that 
$W_k-W_k^{-}$ is not algebraically equivalent to zero
for a generic $X$.
For the Klein quartic $C_7^{1,2}$ and Fermat sextic $F_6$,
we \cite{T1,T2} computed the harmonic volume using the special
values of the generalized hypergeometric function ${}_3F_{2}$
and showed that 
the algebraic cycle
$X-X^-$ is not algebraically equivalent to zero
in $J(X)$.
Recently,  
Otsubo \cite{O} ably extended Harris' and our results,
using a primitive $N$-th root of unity and the trace map
for the Fermat curve $F_N$.
He obtained the algorithm showing that
the algebraic $k$-cycle
$W_k-W_k^-$ is not algebraically equivalent to zero
in $J(F_N)$.
We find the above condition for $N$
and another algorithm showing
that $W_k-W_k^-$ is not algebraically equivalent to zero
in $J(C_N)$.
For a complex algebraic variety $V$,
we define the $p$-th Griffith group $\mathrm{Griff}^{p}(V)$ 
which is generated by
all the algebraic cycles of codimension $p$ in $V$ homologically
equivalent to zero 
modulo algebraic equivalence.
We also prove the Griffiths group of $J(X)$
is nontrivial.
Furthermore, Bloch \cite{B} studied 
the Fermat quartic $F_4$
by means of $L$-functions.

%It is known that the genus of $F_N$ is $(N-1)(N-2)/{2}$.
%We have a question whether or not there exists $X$
%of the genus except for those of $F_N$
%which satisfies the nontrivial condtion.
%We give one answer to the question.
%For example, the genus of $C_{19}$ is
%$9$ and not those of $F_N$.
%Since the value $f(N,1)$ is not integer, we prove that
%$W_1-W_1^-$ is not algebraically equivalent to zero
%in $J(C_{19})$.

We give our method
to prove the algebraic cycle
$C_N-C_N^{-}$ is not algebraically equivalent to zero
in $J(C_N)$,
which is similar to Otsubo's one.
See Hodge's letter \cite[pp.~533--534]{We2}.
Let $\eta_m$ be a third exterior product of
holomorphic $1$-forms on $C_N$.
If the cycle $C_N-C_N^{-}$ is algebraically equivalent to zero
in $J(C_N)$, then
the trace map image $f(N,1)\in \R/{\Z}$
of the harmonic volume at $\eta_m$ 
are zero modulo $\Z$.
In order to prove
the cycle $C_N-C_N^{-}$
is not algebraically equivalent to zero,
we have only to show
the above values are not zero.
Similarly we obtain the method that
$W_k-W_k^{-}$ 
is not algebraically equivalent to zero.

Now we describe the contents of this paper
briefly.
In Section 
\ref{The harmonic volumes},
we introduce the harmonic volume
and relation between it and the Ceresa cycle.
Section \ref{The Fermat curve} is
devoted to definition of the Fermat curve
and the trace map.
In Section \ref{Some values},
we define some quotients of Fermat curve and
recall Otsubo's method.
Using an algebraic condition,
we obtain the harmonic volume $f(N,k)$ of $C_N$.
We carry the numerical computation of
the value by means of the special values of
the generalized hypergeometric function ${}_3F_{2}$.

\noindent
{\bf Acknowledgements.}
The author would like to thank
Noriyuki Otsubo
for his useful comments.
This work is supported by
Grant-in-Aid for Young Scientists (B).

\tableofcontents

\section{The harmonic volume and the algebraic cycle $X-X^{-}$}\label{The harmonic volumes}
\label{Preliminaries1}
We recall the harmonic volume  \cite{H-1} for
a compact Riemann surface $X$ of genus $g \geq 3$. 
We identify
the first integral homology group $H_1(X;\mathbb{Z})$
of $X$ with the first integral cohomology group
by Poincar\'e duality,
and denote it by $H$.
The Hodge star operator $\ast$
on the space of all the $1$-forms $A^1(X)$
is locally given by 
$\ast (f_1(z)dz + f_2(z)d\bar{z})
=-\sqrt{-1}f_1(z)dz + \sqrt{-1}f_2(z)d\bar{z}$
 in a local coordinate $z$
 and depends only on the complex structure
and not on the choice of Hermitian metric.
We identify $H$
with the space of all the
real harmonic $1$-forms on $X$
with integral periods.
Let $(H^{\otimes 2})^{\prime}$ be the kernel of the intersection pairing
$(\ , \ )\colon H\otimes_{\Z} H \to \Z$.
For the rest of this paper, we write $\otimes=\otimes_{\Z}$,
unless otherwise stated.
For any $\sum_{i=1}^{n}a_i\otimes b_i\in (H^{\otimes 2})^{\prime}$,
there exists a unique $\eta\in A^1(X)$
such that $d\eta=\sum_{i=1}^{n}a_i\wedge b_i$
and $\displaystyle \int_{X}\eta\wedge \ast\alpha=0$
for any closed $1$-form $\alpha\in A^{1}(X)$.
Here $a_i$ and $b_i$ are regarded as real harmonic
$1$-forms on $X$.
Choose a point $x_0\in X$.
\begin{defn}\label{pointed harmonic volume}
\mbox{(The pointed harmonic volume \cite{P})}\\
For $\sum_{i=1}^{n}a_i\otimes b_i\in (H^{\otimes 2})^{\prime}$
and $c\in H$,
the pointed harmonic volume $I_{x_0}$ is
the homomorphism $(H^{\otimes 2})^{\prime}\otimes H\to \R/{\Z}$
defined by
$$I_{x_0}{\Biggl(}{\biggl(}\sum_{i=1}^{n}a_{i}\otimes b_{i}{\biggr)}\otimes c{\Biggr)}=\sum_{i=1}^{n}\int_{\gamma}a_{i}b_{i}-\int_{\gamma}\eta
 \quad \mathrm{mod} \ \mathbb{Z}.$$
Here $\eta\in A^1(X)$ is associated to 
$\sum_{i=1}^{n}a_i\otimes b_i$
in the way stated above and
$\gamma$ is a loop in $X$ with the base point $x_0$ whose homology class is
equal to $c$.
The integral $\displaystyle \int_{\gamma}a_ib_i$
is Chen's iterated integral \cite{C}, that is,
$\displaystyle \int_{\gamma}a_ib_i
=\int_{0\leq t_1\leq t_2\leq 1}f_i(t_1)g_i(t_2)dt_1dt_2$
for $\gamma^{\ast}a_i=f_i(t)dt$ and $\gamma^{\ast}b_i=g_i(t)dt$.
Here $t$ is the coordinate in the interval $[0,1]$.
\end{defn}

The harmonic volume is given as a restriction of
the pointed harmonic volume $I_{x_0}$.
We denote by $(H^{\otimes 3})^{\prime}$ the kernel of
a natural homomorphism 
\[H^{\otimes 3} \to H^{\oplus 3}
; a\otimes b\otimes c\mapsto ((a, b)c, (b, c)a, (c, a)b).\]
The {\it harmonic volume} $I$ for $X$
is a linear form on $(H^{\otimes 3})^{\prime}$
with values in $\R/{\Z}$ defined by
the restriction of 
$I_{x_0}$ to $(H^{\otimes 3})^{\prime}$, {\it i.e.},
$I=I_{x_0}|_{(H^{\otimes 3})^{\prime}}$.
Harris \cite{H-1}
proved that the harmonic volume $I$
is independent of the choice of the base point $x_0$.
We denote $\wedge^3 H$ by the third exterior power
of $H$ and $(\wedge^3 H)^{\prime}$
by the kernel of
a homomorphism 
\[\wedge^3 H \to H
; a\wedge b\wedge c\mapsto (a, b)c+(b, c)a+(c, a)b.\]
Then the natural map $(H^{\otimes 3})^{\prime}
\to (\wedge^3 H)^{\prime}$
and $2I$ factors through
\[2I\colon (\wedge^3 H)^{\prime}\to \R/{\Z}\]
\cite{H-1}.

Let $J=J(X)$ and $X_k$ be the Jacobian variety and
$k$-th symmetric product of $X$ respectively.
By the Abel-Jacobi map $X\to J(X)$,
$X_k$ is embedded in $J$.
The image of $X_k$ is denoted by $W_k$.
 The algebraic $k$-cycle $W_k-W_k^{-}$ in $J$
is homologous to zero.
Here we denote by $W_k^{-}$ the image of $W_k$
under the multiplication map by $-1$.
The cycle $W_k-W_k^{-}$ is called 
the $k$-th Ceresa cycle.
We put $W_1-W_1^{-}=X-X^{-}$.
We say the
an  algebraic cycle $1$-cycle $C$ is
{\it algebraically equivalent to zero in} $J$
if there exists a topological $3$-chain
$W$ such that $\partial W=C$
and $W$ lies on $S$,
where $S$
is an algebraic (or complex analytic)
subset of $J$
of complex dimension $2$ (Harris \cite{H-4}).
The chain $W$ is unique up to
$3$-cycles.
We denote by $H^{1,0}$ and $H^{0,1}$
the space of all the
holomorphic and untiholomorphic $1$-forms on $X$
respectively.
From \cite{H-3} and \cite[2.6]{H-4}, we have
\begin{prop}\label{cycle and intermediate Jacobian} 
If $X-X^-$ is algebraically
equivalent to zero in $J$, then $2I(\omega)=0$ modulo $\Z$
for any $\omega\in \wedge^3 H \cap (\wedge^3 H^{1,0}+\wedge^3 H^{0,1})$.
\end{prop}

If the value 
$2I(\omega)$
is nonzero modulo $\Z$ for some
$\omega\in \wedge^3 H \cap (\wedge^3 H^{1,0}+\wedge^3 H^{0,1})$,
then $X-X^-$ is not algebraically equivalent to zero in $J$.

Generally, 
if $W_k-W_k^-$ is algebraically
equivalent to zero in $J$ and satisfying algebraic conditions.
Then a constant multiple of $2I(\omega)$
is equal to $0$ modulo $\Z$
for any $\omega\in \wedge^3 H \cap (\wedge^3 H^{1,0}+\wedge^3 H^{0,1})$.
See Faucette \cite{Fa} and Otsubo \cite{O}.
In particular, Otsubo studied the good condition
for the Fermat curve $F_N$.

\section{The Fermat curve}\label{The Fermat curve}
For $N\in \Z_{\geq 4}$, let $F_N=\{(X:Y:Z)\in\C P^2; X^N+Y^N=Z^N\}$
denote the Fermat curve of degree $N$,
which is a compact Riemann surface of genus $(N-1)(N-2)/{2}$.
Let $x$ and $y$ denote $X/{Z}$ and 
$Y/{Z}$ respectively. 
The equation $X^N+Y^N=Z^N$
induces $x^N+y^N=1$.
Here $\zeta$ denotes $\mathrm{exp}(2\pi\sqrt{-1}/{N})$.
Holomorphic automorphisms $\alpha$ and $\beta$ of $F_N$
are defined by
$\alpha(X:Y:Z)=(\zeta X:Y:Z)$ and $\beta(X:Y:Z)=(X:\zeta Y:Z)$
respectively.
Let $\mu_{N}$ be the group of $N$-th roots of unity in $\C$.
We have that $\alpha\beta=\beta\alpha$ and
 the subgroup of the holomorphic automorphisms of $F_N$
generated by $\alpha$ and $\beta$ is isomorphic to
$\mu_{N}\times \mu_{N}$.
We denoted it by $G$.
Let $\gamma_0$ be a path
$[0,1]\ni t\mapsto 
(t,\sqrt[N]{1-t^N})\in F(N)$,
where $\sqrt[N]{1-t^N}$ is a real nonnegative analytic function
on $[0,1]$.
A loop in $F_N$ is defined by
$$\kappa_0=\gamma_0\cdot(\beta\gamma_0)^{-1}\cdot
(\alpha \beta\gamma_0)\cdot(\alpha\gamma_0)^{-1},$$
where the product $\ell_1\cdot \ell_2$
indicates that we traverse $\ell_1$ first, then $\ell_2$.
We consider a loop $\alpha^i\beta^j\kappa_0$
as an element of the first homology group $H_1(F_N;\Z)$ of $F_N$.
It is a known fact that
$H_1(F_N;\Z)$ is a cyclic $G$-module [Appendix in \cite{Gr}].

Let $\mathbf{I}$ be an index set 
$\{(a,b)\in (\Z/{N\Z})^{\oplus 2}; a,b, a+b\neq 0\}$.
For $a\in \Z/{N\Z}\setminus \{0\}$,
we denote its representative $\langle a\rangle\in \{1,2,\ldots,N-1\}$.
A differential 1-form on $F_N$ is defined by
\[\omega_{0}^{a,b}
=x^{\langle a\rangle -1} y^{\langle b\rangle -1}dx/{y^{N-1}}\]
Set $\mathbf{I}_{\textrm{holo}}
=\{(a,b)\in \mathbf{I}; \langle a\rangle+\langle b\rangle<N\}$.
It is well known that $\{\omega_{0}^{a,b}\}_{\mathbf{I}_{\textrm{holo}}}$
is a basis of $H^{1,0}$ of $F(N)$.
See Lang \cite{L} for example.
It is clear that
\[\int_{\alpha^i\beta^j\gamma_0}\omega_{0}^{a,b}
=\zeta^{ai+bj}\int_{\gamma_0}\omega_{0}^{a,b}
=\zeta^{ai+bj}\frac{B(\langle a\rangle/N,\langle b\rangle/N)}{N}.\]
The beta function $B(u,v)$ is defined by $\displaystyle\int_0^1t^{u-1}
(1-t)^{v-1}dt$ for $u,v>0$.
We denote $B^N_{a,b}=B(\langle a\rangle/N,\langle b\rangle/N)$.
The integral of $\omega_{0}^{a,b}$
along $\alpha^i\beta^j\kappa_0$ is obtained as follows.
\begin{prop}
[Appendix in \cite{Gr}]
\label{periods}
We have
\[\int_{\alpha^i\beta^j\kappa_0}\omega_{0}^{a,b}
=B^N_{a,b}(1-\zeta^a)(1-\zeta^b)\zeta^{ai+bj}/{N}.\]
\end{prop}
We denote the $1$-form
$N\omega_{0}^{a,b}/{B^N_{a,b}}$
by
$\omega^{a,b}$.
This implies
$\displaystyle \int_{\alpha^i\beta^j\kappa_0}\omega_{0}^{a,b}
\in \Z[\zeta]$.

Let $K=\Q(\mu_{N})$ be the $N$-cyclotomic field,
$\mathcal{O}$ be its  integer ring and fix a primitive $N$-th root
of unity $\xi$. 
For a $\Z$-module $M$, we denote the $\mathcal{O}$-module
$M_{\mathcal{O}}=M \otimes \mathcal{O}$.
For each embedding $\sigma\colon K\hookrightarrow \C$,
we may consider the 1-form $\omega^{a,b}$ as an element
of $H_{\mathcal{O}}$
depending on the relation of $\sigma(\xi)$ and $\zeta$.

The harmonic volume naturally extends to
\[2I_{\mathcal{O}}\colon (\wedge^3 H)^{\prime}_{\mathcal{O}}\to 
(\mathcal{O}\otimes  \R)/{\mathcal{O}}.\]
We have the natural isomorphism
\[\mathcal{O}\otimes  \R \cong \left[\prod_{\sigma\colon K\hookrightarrow \C}\C\right]^{+}\]
where $\sigma$ runs through the embedding of $K$ into $\C$ and $+$ denotes the
fixed part by the complex conjugation acting the set $\{\sigma\}$ and $\C$
at the same time. Let $2I_{\sigma}$ denote the $\sigma$-component of $2I$.
Let $\mathop{\rm Tr}\colon (\mathcal{O}\otimes \R)/{\mathcal{O}}\to \R/{\Z}$
be the trace map.
We obtain 
$\mathop{\rm Tr}\circ 2I_{\mathcal{O}}
=\sum_{\sigma\colon K\hookrightarrow \C} 2I_{\sigma}$.
In order to prove the nontriviality of $2I_{\mathcal{O}}$,
it is enough to prove that of $\mathop{\rm Tr}\circ 2I_{\mathcal{O}}$.

\section{Some values of the harmonic volume for the quotient of
Fermat curve}
\label{Some values}
\subsection{Some quotients of Fermat curve}
For a prime number $N$ such that $N\geq 5$,
we define the quotient of Fermat curve
$C^{a,b}_N$ as projective curve whose affine equation is
\[C^{a,b}_N:=\{(u,v)\in \C^2;
v^N=u^a(1-u)^b\}.\]
Here the integers $a,b$ are coprime and satisfy $0<a,b<N$.
It is a compact Riemann surface of genus $(N-1)/2$.
We denote by $\pi\colon F_{N}\to C_{N}$
the $N$-fold unramified covering
$\pi(x,y)=(u,v)=(x^N,x^ay^b)$.
For any integer $h\in \{1,2,\ldots,N-1\}$, there is a unique 1-form
$\eta^{\langle ha\rangle,\langle hb\rangle}$
such that
$\pi^{\ast}\eta^{\langle ha\rangle,\langle hb\rangle}=\omega^{ha,hb}$.
Then we have
$\{\eta^{\langle ha\rangle,\langle hb\rangle}\}
_{\langle ha\rangle+\langle hb\rangle <N}$
is a basis of $H^{1,0}$ of $C_N$.
See Lang \cite{L} for example.

For the rest of this paper,
we assume that the prime number $N$
satisfies $N=1$ modulo $3$.
There exists an integer $1< m < N-1$ such that $m^2+m+1=0$ modulo $N$.
Set $(a_1,b_1)=(1,m), (a_2,b_2)=(m,m^2)$, and $(a_3,b_3)=(m^2,1)$.
\begin{lem}
The above $(a_i,b_i)$'s satisfy the assumption 4.4 in \cite{O}.
Furthermore,
the conditions $(ha_i, hb_i)\in \mathbf{I}_{\textrm{holo}}, i=1,2,3$
are equivalent.
\end{lem}
\begin{proof}
Note that 
$h+\langle hm\rangle+\langle hm^2\rangle
=N$ or $2N$.
We obtain that
$h+\langle hm\rangle+\langle hm^2\rangle
=N$ if only and if
$(ha_i, hb_i)\in \mathbf{I}_{\textrm{holo}}$ for each $i$.
\end{proof}

From now on, we put $C_{N}=C^{1,m}_{N}$.
Since $\pi$ is an $N$-fold unramified covering,
we obtain $N\eta^{\langle ha\rangle,\langle hb\rangle}
\in H_{\mathcal{O}}$ of $C_N$.
In order to compute the harmonic volume
of $C_N$,
it is enough to substitute $N\eta^{ha, hb}$ for
$\varphi^{a,b}$ in \cite{O}.
Set
\[\eta_m
=\dfrac{N\eta^{1,m}\wedge N\eta^{m,\langle m^2\rangle}\wedge
N\eta^{\langle m^2\rangle,1}}{(1-\xi^{-m^2})(1-\xi^{-1})}
.\]
From Proposition \ref{periods},
it is easy to show $\eta_m$ is an element
of $(\wedge^3 H_{\mathcal O})^{\prime}$ of $C_N$.
We have the equation
\[
I_{\mathcal{O}}(\eta_m)
=NI_{\mathcal{O}}(\pi^{\ast}\eta_m)\ \text{mod}\ \mathcal{O}.
\]
Here the harmonic volume of LHS is on $C_N$, and that of RHS
is on $F_{N}$.
Theorem 3.7 in \cite{O} gives us
\begin{prop}
We obtain the value of the harmonic volume for $C_{N}$
\[
 \mathop{\rm Tr}\circ 2I_{\mathcal{O}}
 \left(\eta_m\right)=
 N^6\sum \int_{\kappa_0}\omega^{h,hm}\omega^{hm,hm^2},
\]
where the sum is taken over
$h\in (\Z/{N\Z})^{\ast}$ such that
$(ha_i,hb_i)\in \mathbf{I}_{\textrm{holo}}$.
\end{prop}
\begin{rem}
The conditions
$(ha_i,hb_i)\in \mathbf{I}_{\textrm{holo}}$
and $h +\langle hm\rangle +\langle hm^2\rangle =N$
are equivalent.
Otsubo defined the embedding $\sigma\colon K\hookrightarrow \C$
such that $\sigma(\xi)=\zeta^h$.
\end{rem}

\subsection{Hypergeometric functions and numerical computation}
For the numerical calculation,
we recall the generalized hypergeometric function ${}_3F_2$.
We denote the gamma function
$\Gamma(\tau)=
\displaystyle\int_{0}^{\infty}e^{-t}t^{\tau-1}dt$
for $\tau>0$
and the Pochhammer symbol $(\alpha, n)=\G(\alpha +n)/{\G(\alpha)}$
for any nonnegative integer $n$.
For $x\in\{z\in \C; |z|<1\}$
and $\beta_1,\beta_2\not\in\{0,-1,-2,\ldots\}$,
the generalized hypergeometric function ${}_3F_2$
is defined by
$${}_3F_2{\Big(}
\left.
\begin{array}{c}
\alpha_1,\alpha_2,\alpha_3\\
\beta_1,\beta_2
\end{array}
\right.
;x{\Big)}
=\sum_{n=0}^{\infty}{{(\alpha_1,n)(\alpha_2,n)(\alpha_3,n)}
\over{(\beta_1,n)(\beta_2,n)(1,n)}}\, x^n.
$$
If $\beta_1+\beta_2-\alpha_1-\alpha_2-\alpha_3 >0$,
then the generalized hypergeometric function ${}_3F_2$
 converges when $|x|=1$.
See \cite{Sl} for example.
We denote
\[\Gamma^{N}
\left(\begin{array}{c}
a_1,a_2,\ldots,a_n\\
b_1,b_2,\ldots,b_m
\end{array}
\right)
=\dfrac{\Gamma(a_1/{N})\Gamma(a_2/{N})\cdots \Gamma(a_n/{N})}
{{\Gamma(b_1/{N})\Gamma(b_2/{N})\cdots \Gamma(b_m/{N})}}.\]
Using proposition 5.3 in \cite{O},
we have
\begin{prop}
 \[\int_{\kappa_0}\omega^{h,hm}\omega^{hm,hm^2}
 =\Gamma^{N}
\left(\begin{array}{c}
N-\langle hm\rangle,N-\langle hm^2\rangle\\[4pt]
\langle hm\rangle
\end{array}
\right)^2
{}_3F_2\left( 
\left.
\begin{array}{c}
h/{N},\langle h\rangle /{N},\langle hm^2\rangle/{N}\\
1,1
\end{array}
\right.
;1\right)
\]
for an integer $h$ such that
$h +\langle hm\rangle +\langle hm^2\rangle =N $.
\end{prop}

\begin{thm}\label{Fermat}
For the quotient of Fermat curve $C_{N}$,
if the value
\[2N^6
\underset{\begin{subarray}{c}
0<h<N\\
h +\langle hm\rangle +\langle hm^2\rangle =N
\end{subarray}}{\sum}
\int_{\kappa_0}\omega^{h,hm}\omega^{hm,hm^2}\]
is not equal to zero modulo $\Z$.
Then,
the algebraic cycle $C_N-C_N^-$ is not algebraically
equivalent to zero in $J(C_N)$.
\end{thm}
This value is independent of the choice of $m$,
we denote it by $f(N,1)$.
Furthermore, we set
$f(N,k)=k!\, N^{4k-4}f(N,1)$ for a positive integer $k$.
Using Corollary 4.9 in \cite{O},
it is to show
\begin{thm}\label{Fermat2}
For the quotient of Fermat curve $C_{N}$ and
an integer $k$ such that $1\leq k \leq (N-3)/2$,
if the value $f(N,k)$
is not equal to zero modulo $\Z$.
Then,
the algebraic cycle $W_k-W_k^-$ is not algebraically
equivalent to zero in $J(C_N)$.
\end{thm}

We show the table of the computation of $f(N,1)$
and Mathematica program \cite{W} of $f(N,k)$.
\begin{center}
 \begin{figure}[htbp]
 \begin{minipage}[cbt]{.25\textwidth}
  $\begin{array}{r|r|l}
 N & m & f(N,1)\\ \hline
7& 2& 0.64692\\
13& 3& 0.30390\\
19& 7& 0.15972\\
31& 5& 0.68272\\
37& 10& 0.53833\\
43& 6& 0.94719\\
61& 13& 0.10498\\
67& 29& 0.67834\\
73& 8& 0.67715\\
79& 23& 0.70081\\
97& 35& 0.67120\\
103& 46& 0.20164\\
109& 45& 0.21967\\
127& 19& 0.75140\\
139& 42& 0.89455\\ 
151& 32& 0.20776\\
157& 12& 0.65104\\
163& 58& 0.47898\\
181& 48& 0.68643\\ 
193& 84& 0.65697\\
199& 92& 0.53788\\
211& 14& 0.92477\\
223& 39& 0.14653\\
229& 94& 0.48453\\ 
241& 15& 0.77552\\
271& 28& 0.95322\\
277& 116& 0.88313
 \end{array}$
 \end{minipage}
 \hspace*{.05\textwidth}
 \begin{minipage}[cbt]{.25\textwidth}
  $\begin{array}{r|r|l}
 N & m & f(N,1)\\ \hline
283& 44& 0.97789\\
307& 17& 0.66173\\
313& 98& 0.96320\\
331& 31& 0.88040\\
337& 128& 0.61843\\
349& 122& 0.57242\\
367& 83& 0.70289\\
373& 88& 0.55905\\
379& 51& 0.13144\\
397& 34& 0.54575\\
409& 53& 0.59176\\
421& 20& 0.86406\\
433& 198& 0.085557\\
439& 171& 0.20173\\
457& 133& 0.055143\\
463& 21& 0.24695\\
487& 232& 0.82059\\
499& 139& 0.89265\\
523& 60& 0.12188\\
541& 129& 0.20975\\
547& 40& 0.13131\\
571& 109& 0.86328\\
577& 213& 0.83477\\
601& 24& 0.16953\\
607& 210& 0.27883\\
613& 65& 0.91661\\
619& 252& 0.91440
 \end{array}$
 \end{minipage}
 \hspace*{.05\textwidth}
 \begin{minipage}[cbt]{.25\textwidth}
  $\begin{array}{r|r|l}
 N & m & f(N,1)\\ \hline
631& 43& 0.50662\\
643& 177& 0.72852\\
661& 296& 0.43828\\
673& 255& 0.20495\\
691& 253& 0.58775\\
709& 227& 0.79285\\
727& 281& 0.33854\\
733& 307& 0.12451\\
739& 320& 0.44354\\
751& 72& 0.78711\\
757& 27& 0.10544\\
769& 360& 0.62163\\
787& 379& 0.10082\\
811& 130& 0.17690\\
823& 174& 0.22898\\
829& 125& 0.86872\\
853& 220& 0.57350\\
859& 260& 0.89417\\
877& 282& 0.70117\\
883& 337& 0.26719\\
907& 384& 0.49691\\
919& 52& 0.47589\\
937& 322& 0.94337\\
967& 142& 0.71751\\
991& 113& 0.94086\\
997& 304& 0.79227
 \end{array}$
 \end{minipage}
\caption{Table of the $f(N,1)$}
\label{table}
\end{figure}

\end{center}

The table \ref{table} shows that
the algebraic cycle $C_N-C_N^{-}$
is not algebraically equivalent to zero
in $J(C_N)$ for $N< 1000$ satisfying
the condition.

\begin{figure}[htbp]
\begin{center}
\includegraphics[width=1\textwidth]{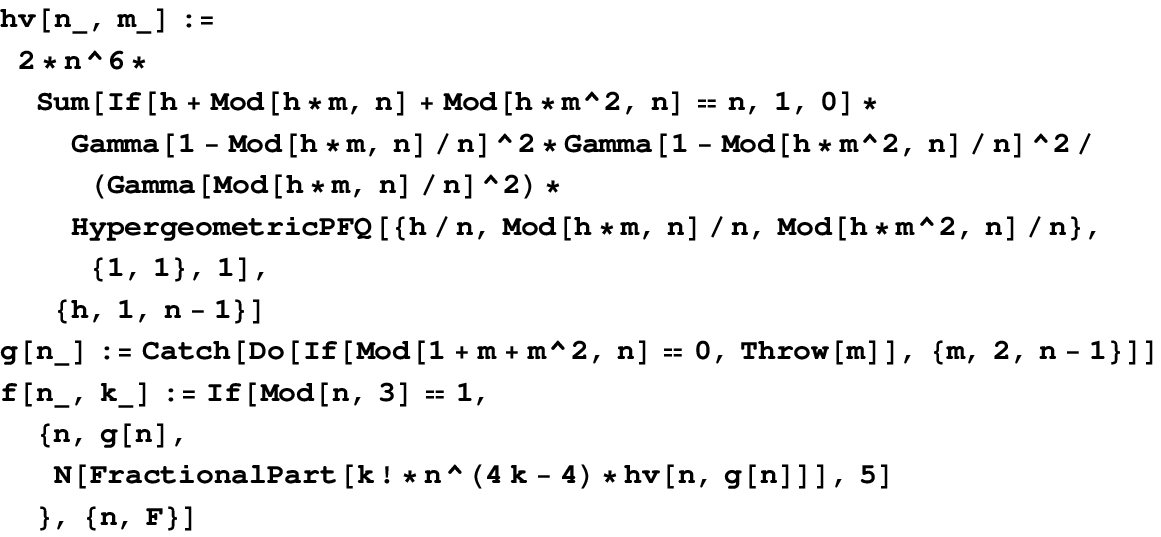}
\caption{Numerical calculation program of $f(N,k)$}
\label{klein-program}
\end{center}
\end{figure}

\end{document}